\documentclass[11pt,reqno]{amsart}
\usepackage{amsthm,amsfonts,amssymb,euscript}

\def\ddb#1{\sqrt{-1}\partial\bar{\partial}#1}
\def\dt#1{\frac{\partial}{\partial t}#1}

\newcommand{\D}{{\mathbb D}^{n}}

\newcommand{\C}{{\mathbb C}^{n}}

\newcommand{\dr}{\omega}
\newcommand{\ka}{K\"{a}hler}

\theoremstyle{plain}
  \newtheorem{theorem}{Theorem}[section]
  \newtheorem{proposition}[theorem]{Proposition}
  \newtheorem{lemma}[theorem]{Lemma}

\theoremstyle{definition}
  \newtheorem{definition}[theorem]{Definition}

\numberwithin{equation}{section}
\begin{document}

\title[Unnormalized conical K\"{a}hler-Ricci flow]{Maximal time existence of unnormalized conical K\"{a}hler-Ricci flow}
\author{Liangming Shen}
\address{Department of Mathematics, Princeton University, Princeton, NJ 08544, USA.}
\email{liangmin@math.princeton.edu}

\begin{abstract}
We generalize the maximal time existence of \ka-Ricci flow in Tian-Zhang \cite{TZ} and Song-Tian \cite{ST3} to conical case.
Furthermore, if the twisted canonical bundle $K_{M}+(1-\beta)[D]$ is big or big and nef, we can expect more on the limit behaviors
of such conical \ka-Ricci flow. Moreover, the results still hold for simple normal crossing divisor.
\end{abstract}

\maketitle


\section{introduction}\label{section1}

    \ka-Ricci flow has become a powerful tool in the study of \ka\ geometry for many years. In \cite{TZ}, Tian-Zhang adapted Tsuji's idea
\cite{Ts} to consider \ka-Ricci flow in the level of cohomology class on projective manifolds. They proved that \ka-Ricci flow can exist
in maximal time interval, where the cohomology class of the \ka\ form $\dr(t)$ stays nondegenerate. Base on this maximal time existence,
they obtained existence results and limit behavior analysis in case that the canonical line bundle is big or big and nef. Later, in the
series of papers \cite{ST1} \cite{ST2} \cite{ST3}, Song-Tian extended this idea to construct an analytic method of minimal model program
in algebraic geometry. \\

    Recently, conic \ka\ metric has become an attracting topic for geometers. More specifically, conic \ka-Einstein metric plays a
critial role in recent great progress of \ka-Einstein problem, see \cite{Ti12} \cite{CDS1} \cite{CDS2} \cite{CDS3}. Besides this seminal
work, there are a lot of papers considering the existence of conical \ka-Einstein metric, for instance, see \cite{Ber} \cite{Br} \cite{CGP}
\cite{GP} \cite{JMR} \cite{LS}. On the other hand, conical \ka-Ricci flow has been studied in \cite{MRS} \cite{CW1} \cite{CW2} \cite{LZ}
\cite{W}. As the study of \ka-Ricci flow, conical \ka-Ricci flow first was considered in hope of evolving conic \ka\ metrics to conical
\ka-Einstein metrics. First, we recall that a \ka\ current $\dr$ is a conical \ka\ metric with angle $2\pi\beta(0<\beta<1)$ along the
divisor D, if $\dr$ is smooth away from D and asymptotically equivalent along D to the model conic metric $$\sqrt{-1}\left(\frac{dz_{1}
\wedge d\bar{z}_{1}}{|z_{1}|^{2(1-\beta)}}+\sum_{i=2}^{n}dz_{i}\wedge d\bar{z}_{i}\right),$$ where $(z_{1},z_{2},\cdots,z_{n})$ are local
holomorphic coordinates and $D=\{z_{1}=0\}$ locally. We call $\dr$ a conic \ka-Einstein metric with angle $2\pi\beta(0<\beta<1)$ along
the divisor D if it is a conic \ka\ metric and satisfies the equation $$Ric(\dr)=\mu\dr+2\pi(1-\beta)[D]$$ in the sense of currents, and
a smooth \ka-Einstein metric outside the divisor D. In \cite{CW1} \cite{CW2} \cite{LZ}, they set $\mu=\beta$ and studied such normalized 
conical \ka-Ricci flow $$\dt\dr=-Ric(\dr)+\beta\dr+2\pi(1-\beta)[D]$$ starting with a conical \ka\ metric with angle $2\pi\beta$ along the 
anticanonical divisor D. In \cite{LZ}, Liu-Zhang used smooth approximation of conic metric, which was set up by Campana-Guenancia-P\^{a}un
in \cite{CGP}, \cite{GP} to obtain a long time solution for the normalized normalized conical \ka-Ricci flow. Moreover, they proved that 
when $\beta\in(0,\frac{1}{2}]$, the flow converges to a conical \ka-Einstein metric if it exists. \\

    Now we can consider whether the conical \ka-Ricci flow exists on general \ka\
manifolds stating with arbitrary conic \ka\ metric, and how long it will last, moreover, what will happen when the flow extincts. Actually,
we can extend the work of Tian-Zhang \cite{TZ} and Song-Tian \cite{ST3} to conic case. And in the study of these problems, we can generalize
the irreducible divisor D to be a simple normal crossing divisor, i.e, $D=\sum\limits_{i=1}^{l}(1-\beta_{i})D_{i}(l\leq n),$ where each
$D_{i}$ is irreducible and $\beta_{i}\in(0,1).$ Locally near a point $p\in\bigcap_{i}D_{i},$ $D_{i}=\{z_{i}=0\}.$ Then a conic metric $\dr$ 
along each $D_{i}$ with cone angle $2\pi\beta_{i}$ is asymptotically equivalent to the model metric $$\dr_{\beta}=\sqrt{-1}\left(\sum_{i=1}
^{l}\frac{dz_{i}\wedge d\bar{z}_{i}}{|z_{i}|^{2(1-\beta_{i})}}+\sum_{i=l+1}^{n}dz_{i}\wedge d\bar{z}_{i}\right).$$
And the unnormalized conical \ka-Ricci flow is
\begin{equation}\label{eq:ckrf-SNC}
 \left\{ \begin{array}{rcl}
&\dt\dr=-Ric(\dr)+2\pi\sum\limits_{i=1}^{l}(1-\beta_{i})[D_{i}]\\
&\dr(0)=\dr^{*}=\dr_{0}+\sum\limits_{i=1}^{l}k\ddb||S_{i}||_{i}^{2\beta_{i}}
        \end{array}\right.
\end{equation}
  
where the initial metric is conic \ka\ metric $\dr^{*}$ with conic angles $2\pi\beta{i}$ along $D_{i}$. Suppose $S_{i}$ is a holomorphic
section of the bundle associated to the divisor $D_{i}$ then this conic \ka\ metric can be written as $\dr^{*}=\dr_{0}+\sum\limits_{i=1}^{l}
k\ddb||S_{i}||_{i}^{2\beta_{i}}$, where $\dr_{0}$ is a smooth \ka\ metric on M, k is a small positive constant and $||\cdot||_{i}^{2}$ is a 
Hermitian metric associated to $[D_{i}]$. Similar to Tian-Zhang and Song-Tian, we can consider this flow equation \eqref{eq:ckrf-SNC} in 
cohomology lever and obtain the following maximal time existence theorem:
\begin{theorem}\label{mainthm1}
 Let $$T_{0}:=\sup\{t|[\dr_{0}]-t(c_{1}(M)-\sum_{i=1}^{l}(1-\beta_{i})[D_{i}])>0\},$$ then starting with a conic \ka\ metric $\dr^{*}$ 
defined above, the unnormalized conical \ka-Ricci flow \eqref{eq:ckrf-SNC} has a unique solution on $[0,T_{0}),$ which is smooth outside the
divisors $D_{i}$ and $C^{2,\alpha}$ along the divisors $D_{i}$.
\end{theorem}
      we can prove this theorem by Tian-Zhang's method \cite{TZ}. However, as the conic metric is not smooth, we need to apply Liu-Zhang's 
approach, making use of the approximation developed in \cite{GP}, to obtain a family of solutions for the approximating equation,
then we prove this family of solutions converge to the solution of the original flow equation. Finally, we prove the solution is smooth
outside the divisors $D_{i}$ and $C^{2,\alpha}$ along the divisor in suitable sense.\\

      We recall that in \cite{TZ} \cite{ST1} \cite{ST2} \cite{ST3}, \ka-Ricci flow is used to study minimal model program. More precisely,
they analysed the behavior of \ka-Ricci flow near the singular time or infinite time on different types of projective manifolds. As we
mentioned before that conic \ka\ metric has becoming an interesting topic, we hope that unnormalized conical \ka-Ricci flow can play some
similar role in the study of conic \ka\ metrics on projective manifolds. We note that conical \ka-Ricci flow always preserves conical
singularities so it's natrual to expect that we can have a good picture of the long time behavior besides conical singularities. Similarly,
if we assume that so-called twisted canonical line bundle $K_{M}+\sum_{i=1}^{l}(1-\beta)[D_{i}])$ is big, we have such a theorem, as 
\cite{TZ},
\begin{theorem}\label{mainthm2}
   If the twisted canonical line bundle is big on the projective manifold M, starting with a conic \ka\ metric $\dr^{*}$,
as t tends to $T_{0}<\infty$ in \ref{mainthm1}, the solution converges to a \ka\ current which is a smooth \ka\ metric outside the divisor
 D and the stable base locus set of $K_{M}+\sum_{i=1}^{l}(1-\beta)[D_{i}])$, say E, and $C^{2,\alpha}$ with conic angle $2\pi\beta_{i}$ 
along $D_{i}\setminus E$.
Moreover, on smooth part, the flow solution converges to the limiting metric in the local $C^{\infty}$-sense and $C^{2,\alpha}$ along
the divisor $D\setminus E$.
\end{theorem}
     Furthermore, if we assume the twisted canonical line bundle $K_{M}+\sum_{i=1}^{l}(1-\beta)[D_{i}])$ is big and nef, by \ref{mainthm1},
we can see that $T_{0}=\infty$. To get a more clear picture of the limiting metric, we can normalized such conical \ka-Ricci flow
equation \eqref{eq:ckrf-SNC} to 
\begin{equation}\label{eq:ckrf*-SNC}
 \dt\dr=-\dr-Ric(\dr)+2\pi\sum_{i=1}^{l}(1-\beta)[D_{i}]).
\end{equation}
We note that the solutions to these two flow equations are different only by scaling. Then similarly we have such a theorem
\begin{theorem}\label{mainthm3}
 All conditions follow \ref{mainthm2}, additionally, if the twisted canonical line bundle $K_{M}+\sum_{i=1}^{l}(1-\beta)[D_{i}])$ is big and
nef, then starting from a conic \ka\ metric $\dr^{*}$, the normalized conical \ka-Ricci flow \eqref{eq:ckrf*-SNC} converges to a current as 
t tends to infinity such that modulo the stable base locus E, the limiting current is a conical \ka-Einstein metric which is $C^{2,\alpha}$
with conic angle $2\pi\beta_{i}$ along $D_{i}\setminus E,$ and independent of the choice of the initial conic metric.
\end{theorem}

         Later we will introduce the definitions of bigness and numerical effectiveness. For the proof, we will
make use of Kodaira's lemma and set up A priori estimates outside the base locus. Note that we still need Liu-Zhang's approach to overcome
the conic singularities.
     For simplicity, in this paper, we only study the conic metrics along only one irreducible divisor D with cone angle $2\pi\beta,$ and
in the end we will briefly discuss how to generalize the proofs to simple normal crossing case. 

\noindent{\bf Acknowledgment.} First the author wants to thank his Ph.D thesis advisor Professor Gang Tian for a lot of discussions and
encouragement. And he also wants to thank Chi Li for many useful conversations. And he also thanks CSC for partial financial
support during his Ph.D career.

\section{maximal time existence}

\subsection{approximation, and $C^{0}$-estimate}

     Now we restrict the situation to $l=1,$ i.e, there is only one irreducible divisor D. Then the unnormalized conical \ka-Ricci flow is     
\begin{equation}\label{eq:ckrf}
\left\{ \begin{array}{rcl}
&\dt\dr=-Ric(\dr)+2\pi(1-\beta)[D]\\&\dr(0)=\dr^{*}=\dr_{0}+k\ddb||S||^{2\beta}
        \end{array}\right.
\end{equation}
As \cite{TZ} \cite{ST3}, first we want to transform the conical \ka-Ricci flow equation \eqref{eq:ckrf} to Monge-Ampere flow equation.
 We set $T_{\delta}=T_{0}-\delta,$ where $T_{0}:=sup\{t|[\dr_{0}]-t(c_{1}(M)-(1-\beta)[D])>0\},$ as we defined before. We know that when
$t\in [0,T_{\epsilon}],$ the cohomology class $[\dr_{0}]-t(c_{1}(M)-(1-\beta)[D])$ is positive. Then when $t\in [0,T_{\delta}],$ there
exist a hermitian metric $||\cdot||$ on the holomorphic line bundle associated to the divisor D and a volume form $\Omega$ on M, such that
\begin{equation}\label{eq:wt}
 \dr_{t}=\dr_{0}-t(Ric(\Omega)-(1-\beta)R(||\cdot||))>0
\end{equation}
where $R(||\cdot||):=-\ddb\log||\cdot||^{2}$ represents the curvature of the
bundle $[D]$. Now we can write $$\overline{\dr_{t}}=\dr_{t}+k\ddb||S||^{2\beta},$$ where k is so small that $\overline{\dr_{t}}$ is positive
as $t\in [0,T_{\delta}].$ Then if we set $\dr=\overline{\dr_{t}}+\ddb\varphi,$ we can write the conical \ka-Ricci flow \eqref{eq:ckrf}
as following: $$\dt(\overline{\dr_{t}}+\ddb\varphi)=\ddb\log(\overline{\dr_{t}}+\ddb\varphi)^{n}+(1-\beta)\ddb\log|S|^{2}.$$
Plug the equation \eqref{eq:wt} into it, we obtain a Monge-Ampere flow equation:
\begin{equation}\label{eq:conic-ma}
 \dt\varphi=\log\frac{(\overline{\dr_{t}}+\ddb\varphi)^{n}}{\Omega}+\log||S||^{2(1-\beta)}
\end{equation}
To approximate the solution to this equation, we first use the smoothing metric in \cite{CGP} that
\begin{equation}\label{eq:appro-conic metric}
 \dr_{t,\epsilon}=\dr_{t}+k\ddb\chi(\epsilon^{2}+||S||^{2}),
\end{equation}
where $$\chi(\epsilon^{2}+t)=\beta\int_{0}^{t}\frac{(\epsilon^{2}+r)^{\beta}-\epsilon^{2\beta}}{r}dr.$$
Before approximating the flow equation, we want to introduce some useful properties of this function $\chi(\epsilon^{2}+t)$. First, we can
find that for each $\epsilon>0,$ $\dr_{t,\epsilon}$ is a smooth \ka\ metric, and as $\epsilon$ tends to 0, $\dr_{t,\epsilon}$ converges to
a conic metric $\overline{\dr_{t}}$ in the sense of currents globally on M and in $C^{\infty}_{loc}$ sense outside D. Moreover, for each
$\epsilon>0,$ this function is smooth, and there exist constants $C>0$ and $\gamma>0$ independent of $\epsilon$ such that as t is finite
$0\leq\chi(\epsilon^{2}+t)\leq C,$ and $\dr_{0,\epsilon}\geq\gamma\dr_{0}.$ As for $t\in[0,T_{\delta}]$ $\dr_{t}$ is comparable to $\dr_{0}$
we can still use a constant $\gamma$ such that $\dr_{t,\epsilon}\geq\gamma\dr_{0}.$\\

    Now we can introduce our approximation equation of \eqref{eq:conic-ma} as following
\begin{equation}\label{eq:conic-ma-appro}
 \left\{ \begin{array}{rcl}
&\dt\varphi_{\epsilon}=\log\frac{(\dr_{t,\epsilon}+\ddb\varphi_{\epsilon})^{n}}{\Omega}
        +\log(||S||^{2}+\epsilon^{2})^{1-\beta}\\&\varphi_{\epsilon}(\cdot,0)=0
        \end{array}\right.
\end{equation}
Equivalently, this equation can be written as the form of generalized \ka-Ricci flow:
\begin{equation}\label{eq:ckrf-appro}
 \left\{ \begin{array}{rcl}
&\dt\dr_{\varphi_{\epsilon}}=-Ric(\dr_{\varphi_{\epsilon}})+(1-\beta)\ddb\log\frac{||S||^{2}+\epsilon^{2}}{||\cdot||^{2}}
\\&\dr_{\varphi_{\epsilon}}(\cdot,0)=\dr_{0,\epsilon}
        \end{array}\right.
\end{equation}
Note that the elliptic version of such approximation equations was set up in \cite{Ti12}. The main steps here are also similar to
Tian's approximation. We need to set up $C^{0}$-estimate, $C^{2}$-estimate and high order derivative estimate step by step and finally
complete approximation argument. In the rest of this subsection we will complete $C^{0}$-estimate and leave other steps to the
next subsection.\\

    Now let's consider the equation \eqref{eq:conic-ma-appro}. We can rewrite this equation as following
$$\dt\varphi_{\epsilon}=\log\frac{(\dr_{t,\epsilon}+\ddb\varphi_{\epsilon})^{n}(||S||^{2}+\epsilon^{2})^{1-\beta}}{\Omega}.$$
To get an upper bound for $\varphi_{\epsilon},$ we can make use of maximal principle to get that
$$\dt\sup\varphi_{\epsilon}\leq\sup\log\frac{\dr_{t,\epsilon}^{n}(||S||^{2}+\epsilon^{2})^{1-\beta}}{\Omega}.$$
Recall the definition of the function $\chi(\epsilon^{2}+t)$, we know from \cite{CGP},
$$\ddb\chi(\epsilon^{2}+||S||^{2})=\frac{\beta^{2}\sqrt{-1}DS\wedge\bar{DS}}{(\epsilon^{2}+||S||^{2})^{1-\beta}}-
\beta((\epsilon^{2}+||S||^{2})^{\beta}-\epsilon^{2\beta})R(||\cdot||).$$
From this computation, we observe that near the divisor D, $$\frac{\dr_{t,\epsilon}^{n}(||S||^{2}+\epsilon^{2})^{1-\beta}}{\Omega}
\approx C,$$ which is independent of $\epsilon.$ So we can get a uniform upper bound that $\varphi_{\epsilon}\leq C_{1}.$
Do the same argument for the lower bound, make use of maximal principle again, we can give a uniform $C^{0}$-estimate that
$$|\varphi_{\epsilon}|\leq C_{1}.$$

   Next we hope to bound the time derivative for $\varphi_{\epsilon}.$ We can apply the technique developed in \cite{ST3}. For simplicity,
we write $\rho=-Ric(\Omega)+(1-\beta)R(||\cdot||)=\dot{\dr}_{t,\epsilon}.$ Take the derivative of \eqref{eq:conic-ma-appro}, we have
\begin{equation}\label{eq:derivative-conic-ma}
 \dt\dot{\varphi}_{\epsilon}=\Delta\dot{\varphi}_{\epsilon}+tr_{\dr}\rho,
\end{equation}
where $\dr=\dr_{t,\epsilon}+\ddb\varphi_{\epsilon},$ and $\Delta$ is the Laplacian w.r.t $\dr.$ First, we compute that
\begin{align*}
 (\dt-\Delta)(t\dot{\varphi}_{\epsilon}-\varphi_{\epsilon}-nt)=&t\,tr_{\dr}\rho+n-tr_{\dr}\dr_{t,\epsilon}-n\\
=&-tr_{\dr}\dr_{0,\epsilon}<0.
\end{align*}
From this we have $\dot{\varphi}_{\epsilon}\leq n+\frac{C_{1}}{t}.$ Combined with the equation \eqref{eq:conic-ma-appro} at time t=0,
and local existence of parabolic equation, we obtain a uniform upper bound for $\dot{\varphi}_{\epsilon}.$
Now we try to derive the lower bound for $\dot{\varphi}_{\epsilon},$
\begin{align*}
 (\dt-\Delta)(\dot{\varphi}_{\epsilon}+A\varphi_{\epsilon}-n\log t)=&tr_{\dr}(\rho+A\dr_{t,\epsilon})+A\log\frac{\dr^{n}(||S||^{2}
+\epsilon^{2})^{1-\beta}}{\Omega}-An-\frac{n}{t}\\ \geq& C(\frac{\dr_{0,\epsilon}^{n}}{\dr^{n}})^{\frac{1}{n}}+A\log\frac{\dr^{n}(||S||^{2}
+\epsilon^{2})^{1-\beta}}{\Omega}-An-\frac{n}{t}\\ \geq& C_{1}(\frac{\dr_{0,\epsilon}^{n}}{\dr^{n}})^{\frac{1}{n}}-\frac{C_{2}}{t}.
\end{align*}
Let's explain these two inequalities. For the first, note that we can choose A sufficiently large such that $\rho+A\dr_{t,\epsilon}\geq
\dr_{0,\epsilon}$ in the time interval, then this inequality follows from Schwarz inequality. For the second, note that in the proof of
$C^{0}$-estimate, we have got that $\frac{\dr_{0,\epsilon}^{n}(||S||^{2}+\epsilon^{2})^{1-\beta}}{\Omega}\approx C,$ then this inequality
follows from the behavior of logarithmic functions. Then by maximal principle, at the minimal point of the function $\dot{\varphi}_
{\epsilon}+A\varphi_{\epsilon}-n\log t,$ we have that $\dr^{n}\geq C_{4}t^{n}\dr_{0,\epsilon}^{n}.$ Make use of this inequality and the
relation of $\dr_{0,\epsilon}$ and $\Omega$ again, we get that
\begin{align*}
\dot{\varphi}_{\epsilon}+A\varphi_{\epsilon}-n\log t&=\log\frac{\dr^{n}(||S||^{2}+\epsilon^{2})^{1-\beta}}{\Omega}+A\varphi_{\epsilon}
-n\log t\\ &\geq\log\frac{C_{4}t^{n}\dr_{0,\epsilon}^{n}(||S||^{2}+\epsilon^{2})^{1-\beta}}{\Omega}+A\varphi_{\epsilon}-n\log t\geq C_{5}.
\end{align*}
As $\varphi_{\epsilon}$ is uniformly bounded, we get a uniform lower bound for $\dot{\varphi}_{\epsilon}.$

\subsection{$C^{2}$-estimate, and high order estimate outside the divisor}

    Now we continue our proof. We want to set up a uniform $C^{2}$-estimate and high order estimate for $\varphi_{\epsilon}.$ First, like
\cite{ST3} \cite{CGP} \cite{LZ}, we prove a Laplacian estimate, which is essentially an application of generalized Schwarz lemma:
\begin{theorem}\label{thm-laplacian}
Let  $\varphi_{\epsilon}$ solve the equation \eqref{eq:conic-ma-appro}. As we have
$$|\varphi_{\epsilon}|\leq C,\,|\dot{\varphi}_{\epsilon}|\leq C,$$ on $[0,T_{\delta}],$ then on this time interval, there exists a uniform
constant A which is independent of $\epsilon,$ such that $$A^{-1}\dr_{0,\epsilon}\leq \dr_{t,\epsilon}+\ddb\varphi_{\epsilon}\leq
A\dr_{0,\epsilon}.$$
\end{theorem}
\begin{proof}
 First we take an holomorphic orthonormal coordinates $(z_{1},\cdots,z_{n})$ at a point p for the metric $\dr_{0,\epsilon}$, say, $g_{0i
\bar{j}}=\delta_{ij},$ such that $g_{i\bar{j}}(\dr_{t,\epsilon})=\lambda_{i}\delta_{ij},\varphi_{\epsilon\,i\bar{j}}=\varphi_{\epsilon\
,i\bar{i}}\delta_{ij}.$ Write $g_{i\bar{j}}$ as the metric
 of $\dr=\dr_{t,\epsilon}+\ddb\varphi_{\epsilon}$ and assume that under such coordinates, $\frac{\partial g_{i\bar{j}}}{\partial z_{k}}=0.$
From approximation flow equations \eqref{eq:conic-ma-appro} \eqref{eq:ckrf-appro}, we compute that
\begin{align*}
 (\dt-\Delta)\log\,tr_{\dr_{0,\epsilon}}\dr&=\frac{(\dt-\Delta)tr_{\dr_{0,\epsilon}}\dr}{tr_{\dr_{0,\epsilon}}}+\frac{|\nabla\,tr_{\dr_{0,
\epsilon}}\dr|^{2}}{|tr_{\dr_{0,\epsilon}}\dr|^{2}}\\&=\frac{1}{tr_{\dr_{0,\epsilon}}\dr}\{g_{0}^{i\bar{i}}(-R_{i\bar{i}}+(1-\beta)
\partial_{i}\partial_{\bar{i}}\log\frac{||S||^{2}+\epsilon^{2}}{||\cdot||^{2}})\\&-g_{i\bar{i}}g^{k\bar{k}}(-R^{i}_{k\bar{k}l}g_{0}^{l
\bar{i}}+R_{i\bar{i}k\bar{k}}(\dr_{0,\epsilon})+\sum\limits_{j}\frac{\partial g_{0i\bar{j}}}{\partial\bar{z_{k}}}\frac{\partial g_{0j
\bar{i}}}{\partial z_{k}})\}+\frac{|\nabla\,tr_{\dr_{0,\epsilon}}\dr|^{2}}{|tr_{\dr_{0,\epsilon}}\dr|^{2}}\\&=\frac{1}{tr_{\dr_{0,
\epsilon}}\dr}\{(1-\beta)\Delta_{\dr_{0,\epsilon}}\log\frac{||S||^{2}+\epsilon^{2}}{||\cdot||^{2}}-g_{i\bar{i}}g^{k\bar{k}}R_{i\bar{i}
k\bar{k}}(\dr_{0,\epsilon})\\&-g_{i\bar{i}}g^{k\bar{k}}\sum\limits_{j}\frac{\partial g_{0i\bar{j}}}{\partial\bar{z_{k}}}\frac{\partial
 g_{0j\bar{i}}}{\partial z_{k}}\}+\frac{1}{|tr_{\dr_{0,\epsilon}}\dr|^{2}}\sum\limits_{i,j}g^{k\bar{k}}g_{i\bar{i}}g_{j\bar{j}}\frac
{\partial g_{0i\bar{i}}}{\partial z_{k}}\frac{\partial g_{0j\bar{j}}}{\partial\bar{z_{k}}}\\&\leq\frac{1}{tr_{\dr_{0,\epsilon}}\dr}
\{(1-\beta)\Delta_{\dr_{0,\epsilon}}\log\frac{||S||^{2}+\epsilon^{2}}{||\cdot||^{2}}-R(\dr_{0,\epsilon})\\&-\sum\limits_{i\leq j}R_{i
\bar{i}j\bar{j}}(\dr_{0,\epsilon})(\frac{\lambda_{i}+\varphi_{\epsilon\,i\bar{i}}}{\lambda_{j}+\varphi_{\epsilon\,j\bar{j}}}+\frac{
\lambda_{j}+\varphi_{\epsilon\,j\bar{j}}}{\lambda_{i}+\varphi_{\epsilon\,i\bar{i}}}-2)\}.
\end{align*}
As \cite{LZ}, we choose a function $\chi_{\rho}(||S||^{2}+\epsilon^{2})=
\rho\int_{0}^{||S||^{2}}\frac{(\epsilon^{2}+r)^{\rho}-\epsilon^{2\rho}}{r}dr.$ By the computation in \cite{CGP} \cite{GP} we have that
$$R_{i\bar{i}j\bar{j}}(\dr_{0,\epsilon})\geq -C_{1}-\frac{C_{2}}{(||S||^{2}+\epsilon^{2})^{1-\beta}}.$$ Meanwhile, we have
$$\ddb\chi_{\rho}(\epsilon^{2}+||S||^{2})=\frac{\rho^{2}\sqrt{-1}DS\wedge\bar{DS}}{(\epsilon^{2}+||S||^{2})^{1-\rho}}-
\rho((\epsilon^{2}+||S||^{2})^{\rho}-\epsilon^{2\rho})R(||\cdot||),$$ so if we choose suitable constants $\rho, C$ we can obtain that
$$R_{i\bar{i}j\bar{j}}(\dr_{0,\epsilon})\geq -C-C'\chi_{\rho i\bar{i}},$$ then we have such estimate
\begin{align*}
 &\frac{1}{tr_{\dr_{0,\epsilon}}\dr}\{-\sum\limits_{i\leq j}R_{i\bar{i}j\bar{j}}(\dr_{0,\epsilon})(\frac{\lambda_{i}+\varphi_{\epsilon
\,i\bar{i}}}{\lambda_{j}+\varphi_{\epsilon\,j\bar{j}}}+\frac{\lambda_{j}+\varphi_{\epsilon\,j\bar{j}}}{\lambda_{i}+\varphi_{\epsilon
\,i\bar{i}}}-2)\}\\ \leq&\frac{1}{\sum\limits_{i}(\lambda_{i}+\varphi_{\epsilon\,i\bar{i}})}\sum\limits_{i<j}\{\frac{\lambda_{i}+
\varphi_{\epsilon\,i\bar{i}}}{\lambda_{j}+\varphi_{\epsilon\,j\bar{j}}}(C+C'\chi_{\rho j\bar{j}})+\frac{\lambda_{j}+\varphi_{\epsilon\,j
\bar{j}}}{\lambda_{i}+\varphi_{\epsilon\,i\bar{i}}}(C+C'\chi_{\rho i\bar{i}})\}\\ \leq&\sum\limits_{i}\frac{C+C'\chi_{\rho i\bar{i}}}
{\lambda_{i}+\varphi_{\epsilon\,i\bar{i}}}=Ctr_{\dr}{\dr_{0,\epsilon}}+C'\Delta\chi_{\rho}.
\end{align*}
Now we want to estimate the other term. First We have that
\begin{align*}
&(1-\beta)\Delta_{\dr_{0,\epsilon}}\log\frac{||S||^{2}+\epsilon^{2}}{||\cdot||^{2}}-R(\dr_{0,\epsilon})\\
=&\Delta_{\dr_{0,\epsilon}}\log\frac{(||S||^{2}+\epsilon^{2})^{1-\beta}\dr_{0,\epsilon}^{n}}
{\Omega}-tr_{\dr_{0,\epsilon}}Ric(\Omega)+(1-\beta)tr_{{\dr_{0,\epsilon}}}R(||\cdot||).
\end{align*}
Now recall that $\dr_{0,\epsilon}\geq\gamma\dr_{0},$ and $Ric(\Omega)$, $R(||\cdot||)$ are uniformly bounded independent of $\epsilon,$
moreover by \cite{CGP} \cite{GP}, $$\ddb\log\frac{(||S||^{2}+\epsilon^{2})^{1-\beta}\dr_{0,\epsilon}^{n}}{\Omega}\leq C\dr_{0,\epsilon}+
C\ddb\chi_{\rho},$$ we can bound this term.
Considering that $tr_{\dr_{0,\epsilon}}\dr\cdot tr_{\dr}\dr_{0,\epsilon}\geq n$, we have
\begin{align*}
&(\dt-\Delta)(\log\,tr_{\dr_{0,\epsilon}}\dr+C'\chi_{\rho}-B\varphi_{\epsilon})\\ \leq&Ctr_{\dr}{\dr_{0,\epsilon}}+Ctr_{\dr}\dr_{0,
\epsilon}-B\dot{\varphi}_{\epsilon}+Bn-Btr_{\dr}\dr_{t,\epsilon}.
\end{align*}
As $\dr_{0,\epsilon}$ and $\dr_{t,\epsilon}$ are equivalent, and $\dot{\varphi}_{\epsilon}$ are uniformly bounded, choose a suitable
 constant B, we obtain that
$$(\dt-\Delta)(\log\,tr_{\dr_{0,\epsilon}}\dr+C'\chi_{\rho}-B\varphi_{\epsilon})\leq C-tr_{\dr}\dr_{0,\epsilon}.$$
By the maximal principle, at the maximal point p of $\log\,tr_{\dr_{0,\epsilon}}\dr+C'\chi_{\rho}-B\varphi_{\epsilon}$ we have
 $tr_{\dr}\dr_{0,\epsilon}(p)\leq C.$ By the approximation flow equation \eqref{eq:conic-ma-appro}, we know that
$$\frac{\dr^{n}}{\dr_{0,\epsilon}^{n}}=e^{\dot{\varphi}_{\epsilon}}\frac{\Omega}{(||S||^{2}+\epsilon^{2})^{1-\beta}\dr_{0,\epsilon}^{n}}$$
is uniformly bounded from above and away from 0, then we obtain that at that point p,
$$tr_{\dr_{0,\epsilon}}\dr(p)\leq\frac{\dr^{n}}{\dr_{0,\epsilon}}(tr_{\dr}\dr_{0,\epsilon})(p)^{n-1}\leq C.$$
As at that point p, $\log\,tr_{\dr_{0,\epsilon}}\dr+C'\chi_{\rho}-B\varphi_{\epsilon}$ attains the maximal, we conclude that on the
whole manifold M, $tr_{\dr_{0,\epsilon}}\dr\leq C$. Now make use of the inequality above again we have that
$$tr_{\dr}\dr_{0,\epsilon}\leq \frac{\dr^{n}}{\dr_{0,\epsilon}^{n}}(tr_{\dr_{0,\epsilon}}\dr)^{n-1}\leq C,$$
which gives us a uniform constant $A>0$ such that
$$A^{-1}\dr_{0,\epsilon}\leq \dr_{t,\epsilon}+\ddb\varphi_{\epsilon}\leq A\dr_{0,\epsilon}.$$
\end{proof}
Now let's consider high order estimates. We know that until now we don't have a uniform $C^{3}$-estimate for conic \ka-Einstein metrics.
However, to get $C^{\infty}$-convergence away from the divisor, we only need to have local high order estimates away from the divisor.
Note that from the two forms of approximation flow equations \eqref{eq:conic-ma-appro} \eqref{eq:ckrf-appro}, compare with \cite{LZ}, we
find that our generalized \ka-Ricci flow has the form that
\begin{equation}\label{eq:gkrf}
 \dt\dr=-Ric(\dr)+\theta,
\end{equation}
where $\theta=(1-\beta)\ddb\log\frac{||S||^{2}+\epsilon^{2}}
{||\cdot||^{2}}$ is uniformly bounded away from the divisor. And we can also define that
$$S=|\nabla_{0}g|_{\dr}^{2}=g^{i\bar{j}}g^{k\bar{l}}g^{p\bar{q}}\nabla_{0i}g_{k\bar{q}}\overline{\nabla}_{0j}g_{p\bar{l}},$$ where g is
the metric of $\dr=\dr_{t}+\ddb\varphi=\dr_{0}-t(Ric(\Omega)-(1-\beta)R(||\cdot||))+\ddb\varphi,$ and $\nabla_{0}$ denotes the covariant
 derivatives with respect to $\dr_{0}.$ Then here we can apply Proposition 2.2 in \cite{LZ} directly:
\begin{theorem}\label{thm-highorder}
 Let $\dr=\dr_{t}+\ddb\varphi$ solve the generalized \ka-Ricci flow \eqref{eq:gkrf} and satisfy
$$A^{-1}\dr_{0}\leq\dr\leq A\dr_{0}\qquad on\qquad B_{r}(p)\times [0,T].$$
Then there exist constant $C',C''$ such that $$S\leq\frac{C'}{r^{2}},\qquad |Rm(\dr)|^{2}\leq\frac{C''}{r^{4}}$$ on $B_{\frac{r}{2}}(p)
\times [0,T]$. Here $C'$ depends on $\dr_{0}, T, A, ||\varphi(\cdot,0)||_{C^{3}(B_{r}(p))}, ||\theta||_{C^{1}(B_{r}(p))}$ and $C''$ depends
on $\dr_{0}, T, A, ||\varphi(\cdot,0)||_{C^{4}(B_{r}(p))}, ||\theta||_{C^{2}(B_{r}(p))}.$
\end{theorem}
Locally outside the divisor, metrics $\dr_{0}, \dr_{0,\epsilon}, \dr_{t,\epsilon}$ are equivalent, and derivatives of these metrics are
also uniformly bounded on the time interval, then we have $C^{3}$-local estimates for $\varphi_{\epsilon}$ under the metric $\dr_{0}.$
By this theorem, using bootstrap methods we can also get local uniform high order estimates outside the divisor D.

\subsection{uniqueness and convergence}
  In this subsection we want to examine how the approximation solutions converge and what the limit likes. First, in the sections above we
proved that on time interval $[0,T_{\delta}]$ we have a priori estimates for the solution of approximation flow equation \eqref
{eq:conic-ma-appro}. Now we hope to solve this equation on $[0,T).$ What remains to do is to prove the solution is independent of the
choice of $\delta.$ We can prove this as \cite{TZ}. We recall that for each $\delta>0,$ we can find a volume form $\Omega$ such that
$\dr_{t}=\dr_{0}-t(Ric(\Omega)-(1-\beta)R(||\cdot||))>0$ on $[0,T_{\delta}].$ Now we suppose that for another $\delta'>0$, we have
$\Omega'$ such that $\dr_{t}'=\dr_{0}-t(Ric(\Omega')-(1-\beta)R(||\cdot||))>0.$ Now we assume that $\Omega'=e^{f}\Omega$ where f is a
smooth function on M and $\varphi_{\epsilon}'$ solves the equation $$\dt\varphi_{\epsilon}'=\log\frac{(\dr_{t}'+\ddb\varphi_{\epsilon}')
^{n}(||S||^{2}+\epsilon^{2})^{1-\beta}}{\Omega'},\qquad \varphi_{\epsilon}'(\cdot,0)=0.$$ Now put $\overline{\varphi_{\epsilon}}=
\varphi_{\epsilon}'+tf,$ as we have $Ric(\Omega')=Ric(e^{f}\Omega)=Ric(\Omega)-\ddb f,$ we compute that
\begin{align*}
 \dt\overline{\varphi_{\epsilon}}=&\dt\varphi_{\epsilon}'+f=\log\frac{(\dr_{t}'+\ddb\varphi_{\epsilon}')^{n}(||S||^{2}+\epsilon^{2})^{1-
\beta}}{\Omega'}+f\\=&\log\frac{(\dr_{t}+t\ddb f+\ddb\varphi_{\epsilon}')^{n}(||S||^{2}+\epsilon^{2})^{1-\beta}}{e^{f}\Omega}+f\\=&
\log\frac{(\dr_{t}+\ddb\overline{\varphi_{\epsilon}})^{n}(||S||^{2}+\epsilon^{2})^{1-\beta}}{\Omega},
\end{align*}
and $\overline{\varphi_{\epsilon}}(\cdot,0)=0.$ From this we find that $\overline{\varphi_{\epsilon}}$ just solves the equation
\eqref{eq:conic-ma-appro}. By the uniqueness of the solution for \eqref{eq:conic-ma-appro} we know $\overline{\varphi_{\epsilon}}=
\varphi_{\epsilon},$ which means that for arbitrary $\delta,\delta'$, their corresponding solutions are the same essentially. So we can
glue these solutions together to get a solution for \eqref{eq:conic-ma-appro} on the time interval $[0,T).$\\

    Finally, we want to prove that as $\epsilon$ tends to 0, the solution $\varphi_{\epsilon}$ to \eqref{eq:conic-ma-appro} converges to a
solution of unnormalized conical \ka-Ricci flow in suitable sense. By the argument above, for simplicity we can prove this on some closed
time interval $[0,T_{\delta}].$ We can prove this as \cite{LZ}. Recall that we have uniform bound for $\varphi_{\epsilon},\dot{\varphi_
{\epsilon}}$ and Laplacian estimate $A^{-1}\dr_{0,\epsilon}\leq \dr_{t,\epsilon}+\ddb\varphi_{\epsilon}\leq A\dr_{0,\epsilon},$ For any
compact set $K\subset\subset M\setminus D$ we have a number N which depends on A and K such that $N^{-1}\dr_{0}\leq \dr\leq N\dr_{0},$ and
on this set K $\log(||S||^{2}+\epsilon^{2})$ converges to $log||S||^{2}$, $\dr_{0,\epsilon}$ converges to $\dr^{*}$ in $C^{\infty}_{loc}$
sense. Let K approximate $M\setminus D$ and $\epsilon_{i}$ tend to 0, on time interval $[0,T_{\delta}],$ by diagonal rule we have a
sequence which is denoted by $\varphi_{\epsilon_{i}}(t)$, and converges to a function $\varphi(t)$, which is smooth away from D, in
$C^{\infty}_{loc}$ sense outside D. By \ref{thm-laplacian} we know that $\dr_{\varphi(t)}=\overline{\dr_{t}}+\ddb\varphi(t)$ are conic
 \ka\ metrics with cone angle $2\pi\beta$ along the divisor D on the whole time interval. \\

Now let's show that the limit potential $\varphi(t)$ we defined above satisfies the conical \ka-Ricci flow equation \eqref{eq:conic-ma}
globally on $M\times [0,T_{\delta}]$ in the sense of currents. For any $(n-1,n-1)$-form $\eta$, by the argument above, as $\epsilon_{i}$
tends to 0, we have
\begin{align*}
 \int_{M}\ddb\dt\varphi_{\epsilon_{i}}\wedge\eta&=\int_{M}\log\frac{(\dr_{\epsilon_{i},t}+\ddb\varphi_{\epsilon_{i}})^{n}(||S||^{2}+
\epsilon_{i}^{2})^{1-\beta}}{\Omega}\wedge\ddb\eta\\&\rightarrow\int_{M}\log\frac{(\overline{\dr_{t}}+\ddb\varphi)^{n}||S||^{2(1-\beta)}}
{\Omega}\wedge\ddb\eta\\&=\int_{M}\ddb\log\frac{(\overline{\dr_{t}}+\ddb\varphi)^{n}||S||^{2(1-\beta)}}{\Omega}\wedge\eta.
\end{align*}
       On the other hand, let $K\subset\subset M\setminus D$ and when K approximates $M\setminus D,$ we have $\int_{M\setminus D}\ddb\eta$
tends to 0. As we have $\dot{\varphi}_{\epsilon_{i}},\dot{\varphi}$ are uniformly bounded independent of $\epsilon,$ we see that
$$|\int_{M}(\dot{\varphi}_{\epsilon_{i}}-\dot{\varphi})\wedge\ddb\eta|\leq|\int_{K}(\dot{\varphi}_{\epsilon_{i}}-\dot{\varphi})\wedge\ddb
\eta|+|\int_{M\setminus K}(\dot{\varphi}_{\epsilon_{i}}-\dot{\varphi})\wedge\ddb\eta|.$$
The first integral tends to 0 as $\varphi_{\epsilon_{i}}$ converges to $\varphi$ in the sense of $C^{\infty}(K),$ and the second one tends
to as $\dot{\varphi}_{\epsilon_{i}},\dot{\varphi}$ are uniformly bounded and K approximates $M\setminus D.$ From above we can see that
$\varphi(t)$ satisfies the conical \ka-Ricci flow equation \eqref{eq:conic-ma} globally on $M\times [0,T)$ in the sense of currents.

     To prove the uniqueness we argue as Wang \cite{W}, which comes from Jeffery's argument \cite{Je}. First we need to verify that
$\varphi(t)$ is H\"{o}lder continuous with respect to
$\dr_{0}.$ As \cite{LZ} write $\phi=\varphi+k||S||^{2\beta}$, in any closed time interval, say $[0,T_{\delta}],$ we have uniform bound for
$\phi,\dot{\phi}$. Then we rewrite the equation of conical \ka-Ricci flow \eqref{eq:conic-ma} as $$(\dr_{t}+\ddb\phi)^{n}=e^{\dot{\phi}}
\frac{\Omega}{||S||^{2(1-\beta)}}$$ away from D. As $\beta\in (0,1)$ we find a constant $\eta$ such that $2(1-\beta)(1+\eta)<2$ such that
$$\int_{M}e^{(1+\eta)(\dot{\phi}-\log||S||^{2(1-\beta)})}\Omega\leq C\int_{M}\frac{\Omega}{||S||^{2(1-\beta)(1+\eta)}}\leq C.$$ By
Kolodziej's $L^{p}$ estimate \cite{Ko} we know that $\varphi(t)$ is H\"{o}lder continuous with respect to $\dr_{0}$ on $[0,T_{\delta}],$
actually, on $[0,T).$ Now suppose there are two solutions $\varphi_{1}(t),\varphi_{2}(t)$ which satisfy all the properties above and
solve the conical \ka-Ricci flow on the maximal time interval, then we set $\tilde{\varphi}_{1}=\varphi_{1}(t)+a||S||^{2p}$ and write
$v=\tilde{\varphi}_{1}-\varphi_{2}.$ Compare corresponding conical \ka-Ricci flow equations, we have that
\begin{equation}\label{eq:uniqueness}
\dt v=\dt(\varphi_{1}-\varphi_{2})=\log\frac{(\overline{\dr_{t}}+\ddb\varphi_{1})^{n}}{(\overline{\dr_{t}}+\ddb\varphi_{2})^{n}}
=\underline{\Delta}v-a\underline{\Delta}||S||^{2p},
\end{equation}
where $$\underline{\Delta}:=\int_{0}^{1}(\overline{g}_{t}+\ddb((1-s)\varphi_{1}+s\varphi_{2}))^{i\bar{j}}\partial_{i}\bar{\partial}_{j}ds.$$
Outside the divisor D we can compute that
\begin{align*}
\ddb||S||^{2p}&=p||S||^{2p}\ddb||S||^{2}+\sqrt{-1}p^{2}||S||^{2p-2}\partial||S||^{2}\wedge\bar{\partial}||S||^{2}\\
&\geq -p||S||^{2p}R(||\cdot||).
\end{align*}
By the estimates above, we know that the metric $\overline{g}_{t}+\ddb((1-s)\varphi_{1}+s\varphi_{2})$ is equivalent to the initial conic
metric $\dr^{*}$ and $C^{\alpha}$ outside the divisor D. Then from \eqref{eq:uniqueness} we obtain that
$$\dt v\leq \underline{\Delta}v+aC,$$ by maximal principle, we have $$v(t)\leq\sup v(0)+aCt=a(Ct+||S||^{2p}).$$ Let a tend to 0, we have
$\varphi_{1}\leq\varphi_{2},$ and we can prove $\varphi_{2}\leq\varphi_{1}$ by the same argument. Finally we obtain the uniqueness of
the limit solution.

\subsection{$C^{2,\alpha}$-estimate on the divisor D}
      To complete the proof of Theorem \ref{mainthm1}, we need to give a $C^{2,\alpha}$-estimate for $\varphi(t),$ as this estimate allows
us to apply inverse function theorem to obtain the existence of the solution to conical \ka-Ricci flow \eqref{eq:conic-ma}. In this part,
we will only introduce the definition of these norms under conic setting and leave the proof to the next paper \cite{Sh}
      Let's describe the basic construction. First we can fix a unit polydisk $\D\subset\C$ with the origin as the center, and the divisor
$D=\{z_{1}=0\}.$ Then the standard conic metric attached with $(\C,D)$ is that $$\dr_{\beta}:=\sqrt{-1}\left(\frac{dz_{1}\wedge d\bar
{z}_{1}}{|z_{1}|^{2(1-\beta)}}+\sum_{i=2}^{n}dz_{i}\wedge d\bar{z}_{i}\right),$$ which defines a Riemannian metric $g_{\beta}$ and induces
a distance $d_{\beta}.$ For simplicity, define $\overline{d}(p_{1},p_{2}):=d_{\beta}(x_{1},x_{2})+|t_{1}-t_{2}|^{\frac{1}{2}}$
as the spacetime distance, where $p_{i}=(x_{i},t_{i}).$ Now for a locally integrable function on $\D\times [0,T)$ and $\alpha\in(0,1),$
we define the H\"{o}lder norm as \cite{GP}:
$$[f]_{\alpha}:=\sup_{\D\times [0,T)}|f|+\sup_{p_{1}\neq p_{2}}\frac{|f(p_{1})-f(p_{2})|}{\overline{d}(p_{1},p_{2})^{\alpha}}.$$ We say
 the function f is $C^{\alpha}$ if the norm $[f]_{\alpha}<+\infty.$ Consider
the vector fields $\xi_{1}=|z_{1}|^{1-\beta}\frac{\partial}{\partial z_{1}},\xi_{k}=\frac{\partial}{\partial z_{k}},$ for $k=2,\cdots n,$
then we say a $(1,0)$-form $\tau$ is $C^{\alpha}$ if $\tau(\xi_{i})]$ is $C^{\alpha}$ for any $i=1,\cdots n,$ and say a $(1,1)$-form
$\sigma$ is $C^{\alpha}$ if $\sigma(\xi_{i},\xi_{j})$ is $C^{\alpha}$ for any $i,j=1,\cdots n.$ And we say f is $C^{2,\alpha}$ if $f,
\partial f, \ddb f, \dot{f}$ are all $C^{\alpha}.$
      
      Consider the equation \eqref{eq:conic-ma}, as $\dr_{t}$ is always smooth on the maximal interval, when r is sufficient small we
can write $\overline{\dr}_{t}+\ddb\varphi(t)=\ddb u$ locally, then we only need to show that u is $C^{2,\alpha}.$ By computation before
we know that $Ric(\dr_{beta})=2\pi(1-\beta)[D]-\ddb F$ for some smooth function f, so the equation for u can be written as $$\dt u=\log
\frac{(\ddb u)^{n}}{\dr_{\beta}^{n}}+F.$$ Take the covariant derivatives with respect to $\dr_{\beta},$ we have
$$\dt u_{k\bar{l}}=u^{i\bar{j}}u_{k\bar{l}i\bar{j}}-u^{i\bar{q}}u^{p\bar{j}}u_{i\bar{j}k}u_{p\bar{q}\bar{l}}+F_{k\bar{l}}.$$

       Unfortunally, classic Evans-Krylov-Safanov estimate for nonlinear PDE can't be used in conic setting. In conic \ka-Einstein case,
as the Ricci curvature of target metric has lower bound, by Cheeger-Colding-Tian theory \cite{CCT}, the metric near the conic singularities
is close to a standard flat conic metric, which indicates that perturbed Schauder estimate can be applied. Alternately, Tian extended
his master thesis to conic case and gave a proof of $C^{2,\alpha}$-estimate in case of conic \ka-Einstein metrics. In \cite{Sh}, we will
follow Tian's approach to prove this theorem:

\begin{theorem}\label{thm-regularity}
 u is $C^{2,\alpha}$-bounded on the divisor D.
\end{theorem}

       This theorem completes the proof for the first main theorem \ref{mainthm1}.

\section{limit behavior near the singular time when the twisted canonical bundle is big}
     In the last section we discussed how long the unnormalized conical \ka-Ricci flow would last. Now we want to understand what will
happen when the flow runs to singular time. During the remaining part of this article  we assume that M is an n-dimensional projective
manifold. In \cite{TZ} \cite{ST3}, to study minimal model program, they mainly considered the case that the canonical line bundle $K_{M}$
is big, or big and nef (numerical effective). First let's recall their definitions:
\begin{definition}
 Suppose L is a holomorphic line bundle on a projective manifold M, then we say L is big if $[c_{1}(L)]^{n}=\int_{M}c_{1}(L)^n>0.$ And
we say L is nef if for any algebraic curve C on M, $c_{1}(L)(C)=\int_{C}c_{1}(L)\geq 0.$
\end{definition}
To study the limit behavior near the singular time, they mainly made use of so-called Kodaira's lemma, which was proved by Kawamata and
used in degenerate Monge-Ampere equations by Tsuji \cite{Ts} first:
\begin{lemma}\label{kodaira}
 Let L be a line bundle on a projective manifold M. If L is big, then there exists an effective line bundle E and two positive numbers
$a,b$ such that $L-\delta E$ is positive for any $\delta\in (a,b).$ Moreover, if L is big and nef, then the conclusion above is true
for $\delta\in (0,b).$
\end{lemma}
Now we can begin the proof of theorem \ref{mainthm2}. As the last section, we can do similar constructions and consider the approximation
flow equation \eqref{eq:conic-ma-appro} again that
$$\dt\varphi_{\epsilon}=\log\frac{(\dr_{t,\epsilon}+\ddb\varphi_{\epsilon})^{n}}{\Omega}+\log(||S||^{2}+\epsilon^{2})^{1-\beta}.$$
Although $\dr_{t,\epsilon}$ may not be a \ka\ metric near the singular time T, as it's controlled from above, we can still make use of
maximal principle to get a uniformly upper bound for $\varphi_{\epsilon}.$ To get a lower bound, we need to apply Kodaira's lemma to
get a family of regular background metrics. Recall that we define $\dr_{t}=\dr_{0}-t(Ric(\Omega)-(1-\beta)R(||\cdot||)),$ and we assume
that the twisted canonical bundle $K_{M}+(1-\beta)[D]$ is big, then we can easily see that the current $[\dr_{t}]$ is big on the time
interval $[0,T_{0}].$ By Kodaira's lemma, we have a $\delta\in (a,b)$ and an effective divisor E with a Hermitian metric $||\cdot||'$ such
that $\dr_{t}+\delta\ddb\log||\cdot||'^{2}>0.$ If we define $S'$ as a local holomorphic section for E, we will have
$$\dr_{t}+\delta\ddb\log||S'||'^{2}>0.$$

Set $\dr'_{t,\epsilon}=\dr_{t,\epsilon}+\delta\ddb\log||S'||'^{2},$ we can rewrite the equation \eqref{eq:conic-ma-appro} as
\begin{equation}\label{eq:big-app}
 \dt(\varphi_{\epsilon}-\delta\log||S'||'^{2})=\log\frac{(\dr'_{t,\epsilon}+\ddb(\varphi_{\epsilon}-\delta\log||S'||'^{2}))^{n}
(||S||^{2}+\epsilon^{2})^{1-\beta}}{\Omega}.
\end{equation}
Note that in this equation, $\dr'_{t,\epsilon}$ is a smooth metric, by the argument in the last section, we know that
$$\log\frac{\dr'^{n}_{t,\epsilon}(||S||^{2}+\epsilon^{2})^{1-\beta}}{\Omega}$$ is bounded from below by a constant depending on $\delta.$
By maximal principle, we know that $\varphi_{\epsilon}-\delta\log||S'||'^{2}$ is bounded from below by $C_{\delta}.$ Now we get a
$C^{0}$-estimate for $\varphi_{\epsilon}$ that $$C_{\delta}+\delta\log||S'||'^{2}\leq\varphi_{\epsilon}\leq C.$$
Now let's begin to estimate $\dot{\varphi}_{\epsilon}.$ For simplicity, we denote $\varphi_{\epsilon,\delta}=\varphi_{\epsilon}-\delta
\log||S'||'^{2},$ and $\dr=\dr'_{t,\epsilon}+\ddb\varphi_{\epsilon,\delta}.$ Then the equation \eqref{eq:big-app} can be written as
$$\dt\varphi_{\epsilon,\delta}=\log\frac{\dr^{n}(||S||^{2}+\epsilon^{2})^{1-\beta}}{\Omega}.$$ Take the derivative of this equation, we
obtain that $$\dt\dot{\varphi}_{\epsilon,\delta}=\Delta\dot{\varphi}_{\epsilon,\delta}+tr_{\dr}\rho,$$ which is the same with
\eqref{eq:derivative-conic-ma}. Note that $\ddb\log||S'||'^{2}=-R(||\cdot||')$ outside the base divisor E, and $C_{\delta}\leq
\varphi_{\epsilon,\delta}\leq C-\delta\log||S'||'^{2},$ we can compute as \cite{ST3},
\begin{align*}
 &(\dt-\Delta)(\dot{\varphi}_{\epsilon,\delta}-A^{2}\varphi_{\epsilon,\delta}+A\log||S'||'^{2})\\=&tr_{\dr}\rho-A^{2}\dot{\varphi}_
{\epsilon,\delta}+A^{2}(n-tr_{\dr}\dr'_{t,\epsilon})+Atr_{\dr}R(||\cdot||')\\=&tr_{\dr}(\rho-A^{2}\dr'_{t,\epsilon}+AR(||\cdot||'))-A^{2}
(\dot{\varphi}_{\epsilon,\delta}-A^{2}\varphi_{\epsilon,\delta}+A\log||S'||'^{2})\\&+(nA^{2}-A^{4}\varphi_{\epsilon,\delta}+A^{3}
\log||S'||'^{2}).
\end{align*}
Choose suitable large constant A such that the first term is negative, make use of $C^{0}$-estimate for $\varphi_{\epsilon,\delta}$ and
$\log||S'||'^{2}$ is bounded from above, we have that
\begin{align*}
(\dt-\Delta)(\dot{\varphi}_{\epsilon,\delta}-A^{2}\varphi_{\epsilon,\delta}+A\log||S'||'^{2})&\leq -A^{2}(\dot{\varphi}_{\epsilon,
\delta}-A^{2}\varphi_{\epsilon,\delta}+A\log||S'||'^{2})\\&+C-C'\log||S'||'^{2}.
\end{align*}
Now make use of maximal principle, note that the maximal of $\dot{\varphi}_{\epsilon,\delta}-A^{2}\varphi_{\epsilon,\delta}+A
\log||S'||'^{2}$ can only be obtained outside E, we can finally get that
$$\dot{\varphi}_{\epsilon,\delta}\leq C-C'\log||S'||'^{2},$$ where the constants $C,C'$ may depend on $\delta.$
Similarly, to get a lower bound, we can compute that
\begin{align*}
 &(\dt-\Delta)(\dot{\varphi}_{\epsilon,\delta}+A^{2}\varphi_{\epsilon,\delta}-A\log||S'||'^{2})\\=&tr_{\dr}\rho+A^{2}\log\frac{\dr^{n}
(||S||^{2}+\epsilon^{2})^{1-\beta}}{\Omega}-nA^{2}+A^{2}tr_{\dr}\dr'_{t,\epsilon}-Atr_{\dr}R(||\cdot||')\\=&tr_{\dr}(\rho+\frac{A^{2}}{2}
\dr'_{t,\epsilon}-AR(||\cdot||'))-nA^{2}\\&+A^{2}(\log\frac{\dr^{n}}{\dr'^{n}_{t,\epsilon}}+\frac{tr_{\dr}\dr'_{t,\epsilon}}{2}
+\log\frac{\dr'^{n}_{t,\epsilon}(||S||^{2}+\epsilon^{2})^{1-\beta}}{\Omega}).
\end{align*}
Choose A large enough such that the first term becomes positive, then make use of Schwarz inequality and the property of logarithmic
functions, we obtain that
\begin{align*}
 &(\dt-\Delta)(\dot{\varphi}_{\epsilon,\delta}+A^{2}\varphi_{\epsilon,\delta}-A\log||S'||'^{2})\\ \geq& -nA^{2}+A^{2}(\log\frac{\dr^{n}}
{\dr'^{n}_{t,\epsilon}}+\frac{1}{2}(\frac{\dr'^{n}_{t,\epsilon}}{\dr^{n}})^{\frac{1}{n}}+\log\frac
{\dr'^{n}_{t,\epsilon}(||S||^{2}+\epsilon^{2})^{1-\beta}}{\Omega})\\ \geq&A^{2}(-\log\frac{\dr^{n}}{\dr'^{n}_{t,\epsilon}}+\log\frac
{\dr'^{n}_{t,\epsilon}(||S||^{2}+\epsilon^{2})^{1-\beta}}{\Omega})-C_{1}\\=&-A^{2}(\dot{\varphi}_{\epsilon,\delta}+A^{2}\varphi_{
\epsilon,\delta}-A\log||S'||'^{2})+2A^{2}\log\frac{\dr'^{n}_{t,\epsilon}(||S||^{2}+\epsilon^{2})^{1-\beta}}{\Omega}\\&+A^{4}\varphi_{
\epsilon,\delta}-A^{3}\log||S'||'^{2}-C_{1}\\ \geq&-A^{2}(\dot{\varphi}_{\epsilon,\delta}+A^{2}\varphi_{\epsilon,\delta}
-A\log||S'||'^{2})-C_{2},
\end{align*}
By maximal principle and argue as above, we can get a lower bound for $\dot{\varphi}_{\epsilon,\delta}$ that
$$\dot{\varphi}_{\epsilon,\delta}\leq -C+C'\log||S'||'^{2}.$$ Now we can summarize these estimates and obtain such a theorem:

\begin{theorem}\label{thm-big-c0}
Suppose the twisted canonical bundle $K_{M}+(1-\beta)[D]$ is big on a projective manifold M, and E is an effective divisor
such that $[\dr_{t}]-\delta[E]=[\dr_{0}]-t[c_{1}(M)-(1-\beta)c_{1}(D)]-\delta[E]>0$ on the time interval $[0,T_{0}]$ for a positive
$\delta\in(a,b).$ Let $\varphi_{\epsilon}$ solves the approximation flow equation \eqref{eq:conic-ma-appro}, then we have that
$$C_{\delta}+\delta\log||S'||'^{2}\leq\varphi_{\epsilon}\leq C,\quad |\dot{\varphi}_{\epsilon}|\leq C-C'\log||S'||'^{2}.$$
\end{theorem}

      Now we can derive a Laplacian estimate as before, which is also an application of generalized Schwarz Lemma.
\begin{theorem}\label{thm-big-c2}
 On the time interval $[0,T_{0}]$ there exist constants $C,\alpha>0$ such that $tr_{\dr_{0,\epsilon}}\dr\leq C||S'||'^{-2\alpha}.$
\end{theorem}
\begin{proof}
 As theorem \ref{thm-laplacian}, we first compute that $$(\dt-\Delta)(\log\,tr_{\dr_{0,\epsilon}}\dr+C'\chi_{\rho})\leq Ctr_{\dr}
{\dr_{0,\epsilon}}.$$ As in the proof of theorem \ref{thm-big-c0}, we define
$$H=\log\,tr_{\dr_{0,\epsilon}}\dr+C'\chi_{\rho}-A^{2}\varphi_{\epsilon,\delta}+A\log||S'||'^{2},$$
then for sufficiently large $A>0,$ we have that
\begin{align*}
 &(\dt-\Delta)H\\ \leq& tr_{\dr}(C\dr_{0,\epsilon}-A^{2}\dr'_{t,\epsilon}+AR(||\cdot||'))-A^{2}\dot{\varphi}_{\epsilon,\delta}
+nA^{2}\\ \leq& -tr_{\dr}\dr_{0,\epsilon}+C-C'\log||S'||'^{2}.
\end{align*}
By maximal principle, we know that the maximal of H can only be obtained outside the divisor E, and when this maximal is obtained,
we have $$tr_{\dr}\dr_{0,\epsilon}\leq C-C'\log||S'||'^{2},$$ which indicates that at the maximal of H,
$$tr_{\dr_{0,\epsilon}}\dr\leq\frac{\dr^{n}}{\dr^{n}_{0,\epsilon}}(tr_{\dr}{\dr_{0,\epsilon}})^{n-1}=e^{\dot{\varphi}_{\epsilon}}
\frac{\Omega}{(||S||^{2}+\epsilon^{2})^{1-\beta}\dr_{0,\epsilon}^{n}}\leq C||S'||'^{-2\alpha'}.$$
Then the theorem follows.
\end{proof}
As in the last section, we can give local high order estimate for $\varphi_{\epsilon}$ outside the divisors D and E.
\begin{proposition}\label{prop-big-high}
 For any $K\subset\subset M\setminus(D\bigcup E)$ and $k>0,$ there exists $C_{k,K}$ such that
$$||\varphi_{\epsilon}||_{C^{k}(K\times [0,T_{0}])}\leq C_{k,K}.$$
\end{proposition}
       By this proposition, we can argue as the last section. We put a sequence of compact set to approximate the regular part $M
\setminus(D\bigcup E).$ On each compact set K, as we have uniform estimates for high order derivatives, we can choose a sequence
$\epsilon_{i}$ such that $\varphi_{\epsilon_{i}}$ converges on $K\times [0,T_{0}]$ in $C_{loc}^{\infty}$ sense. By diagonal argument,
we can take a sequence, say, $\varphi_{\epsilon_{i}}$ converges to a function $\varphi$ on $(M\setminus(D\bigcup E))\times [0,T_{0}]$
$C_{loc}^{\infty}$ sense, which is smooth on $M\setminus(D\bigcup E).$ By theorem \ref{thm-big-c2} we know that
$$\dr_{\varphi}=\dr_{t,\epsilon}+\ddb\varphi$$ is a conic \ka\ metric with angle $2\pi\beta$ along the part of divisor D outside the
divisor E on $[0,T_{0}].$ Now as above, we want to prove that $\varphi_{\epsilon_{i}}$ converges to $\varphi$ globally in the sense of
currents. Other arguments are the same except for the part near the divisor E. Note that when we consider the integral
$$\int_{M}\log\frac{(\overline{\dr_{t}}+\ddb\varphi)^{n}||S||^{2(1-\beta)}}{\Omega}\wedge\ddb\eta,$$ near the divisor E, by theorem
\ref{thm-big-c2}, the integral function grows as $\log||S||',$ so this integral exists. On the other hand, by theorem \ref{thm-big-c2},
as $\dot{\varphi},\dot{\varphi}_{\epsilon_{i}}$ grows like $\log||S||',$ then it's also true that
$$|\int_{M}(\dot{\varphi}_{\epsilon_{i}}-\dot{\varphi})\wedge\ddb\eta|$$ tends to 0. So we know that the limit $\varphi(t)$ outside the
divisor E satisfies conical \ka-Ricci flow equation on $[0,T_{0}]$ in the sense of currents, and as t tends to $T_{0},$ $\dr_{\varphi}=
\overline{\dr_{t}}+\ddb\varphi(t)$ converges smoothly to a \ka\ metric on $M\setminus(D\bigcup E),$ and conic along $D\setminus E$ with
cone angle $2\pi\beta.$ Now we want to claim that the limit is unique, i.e. for any sequence $\{\epsilon_{i}\}$ which tends to 0, the
sequence $\varphi_{\epsilon_{i}}$ will converge to the same limit in our sense. Actually this follows from the uniqueness of the
solution on the time interval $[0,T_{0}).$ Then consider that in any compact set in $M\setminus(D\bigcup E),$ as $\dot{\varphi}_{
\epsilon_{i}}$ is uniformly bounded, we can conclude that the limit of any sequence $\varphi_{\epsilon_{i}}$ converges to the same
limit, say, $\varphi(t)$ at $T_{0},$ which completes our claim. Finally, note that the base locus E may not be unique, however, we can
consider the intersection of these E's and call it the stable base locus, and still denote it as E. By the argument above, we can
obtain a unique limit current $\dr_{\varphi}=\dr_{t,\epsilon}+\ddb\varphi$ outside E at time $T_{0},$ and this current is a smooth
\ka\ metric outside $D\bigcup E,$ and conic along $D\setminus E$ with cone angle $2\pi\beta.$ Finally, $C^{2,\alpha}$-estimate in the
proof of theorem \ref{mainthm1} still holds here. Now we complete the proof of theorem \ref{mainthm2}.

\section{limit in the case of big and nef twisted canonical bundle}

     In this section we want to understand the limit behavior when the twisted canonical bundle is big and nef. By theorem \ref{mainthm1}
we know that this conical \ka-Ricci flow exists forever. In this case we find that $\dr_{t}=\dr_{0}-t(Ric(\Omega)-(1-\beta)R(||\cdot||))$
will blow up at infinite time. To capture the limit behavior more precisely, we can consider normalized conical \ka-Ricci flow
\begin{equation}\label{eq:ckrf**}
 \dt\dr=-\dr-Ric(\dr)+2\pi(1-\beta)[D].
\end{equation}
Note that if we define $\tilde{\dr}(t):=(1+t)\dr(\log(1+t)),$ then we can compute that
\begin{align*}
 \dt\tilde{\dr}(t)=&\dr(\log(1+t))+(-Ric(\dr(\log(1+t))-\dr(\log(1+t))+2\pi(1-\beta)[D]))\\
=&-Ric(\tilde{\dr}(t))+2\pi(1-\beta)[D],
\end{align*}
which means that the solution of unnormalized conical \ka-Ricci flow is actually the same with normalized conical \ka-Ricci flow
only moduli a scaling. As we did in the proof of theorem \ref{mainthm1}, we can consider this equation in the level of cohomology class.
Then we can write
$$\dr_{t}=e^{-t}\dr_{0}+(1-e^{-t})(-Ric(\Omega)+(1-\beta)R(||\cdot||)),$$ and set $\overline{\dr_{t}}=\dr_{t}+k\ddb||S||^{2\beta}$ and
$\dr=\overline{\dr_{t}}+\ddb\varphi.$ Now we can write the equation \eqref{eq:ckrf**} in scalar form:
\begin{equation}\label{eq:ma-nef}
  \left\{ \begin{array}{rcl}
&\dt\varphi=\log\frac{(\overline{\dr_{t}}+\ddb\varphi)^{n}}{\Omega}+\log||S||^{2(1-\beta)}-\varphi-k||S||^{2\beta}
\\&\varphi(\cdot,0)=0
        \end{array}\right.
\end{equation}
    As before, we can consider an approximation flow equation of \eqref{eq:ma-nef}
\begin{equation}\label{eq:app-ma-nef}
 \left\{ \begin{array}{rcl}
&\dt\varphi_{\epsilon}=\log\frac{(\dr_{t,\epsilon}+\ddb\varphi_{\epsilon})^{n}(||S||^{2}+\epsilon^{2})^{1-\beta}}{\Omega}
-\varphi_{\epsilon}-k\chi(\epsilon^{2}+||S||^{2})
\\&\varphi_{\epsilon}(\cdot,0)=0
        \end{array}\right.
\end{equation}
where $$\dr_{t,\epsilon}:=e^{-t}\dr_{0}+(1-e^{-t})(-Ric(\Omega)+(1-\beta)R(||\cdot||))+k\ddb\chi(\epsilon^{2}+||S||^{2})$$ as before.\\

      First consider $C^{0}$-estimate for $\varphi_{\epsilon}.$ By maximal principle, as $$\log\frac{\dr_{t,\epsilon}^{n}(||S||^{2}+
\epsilon^{2})^{1-\beta}}{\Omega}-k\chi(\epsilon^{2}+||S||^{2})$$ is bounded from above, it's easy to see that $\varphi_{\epsilon}\leq C.$
 On the other hand, we recall lemma \ref{kodaira}as the twisted canonical bundle $K_{M}+(1-\beta)[D]$ is big
and nef, there exists an effective divisor E such that $\dr_{t,\epsilon}+\delta\ddb\log||S'||'^{2}>0$ for all $t\in[0,\infty)$ and any
$\delta\in(0,a),$ where $S'$ is a local defining section of E. Then the equation \eqref{eq:app-ma-nef} can be written as
\begin{align*}
\dt(\varphi_{\epsilon}&-\delta\ddb\log||S'||'^{2})=\log\frac{(\dr'_{t,\epsilon}+\ddb(\varphi_{\epsilon}-\delta\ddb\log||S'||'^
{2}))^{n}}{\Omega||S'||'^{2\delta}}\\&+\log(||S||^{2}+\epsilon^{2})^{1-\beta}-(\varphi_{\epsilon}-\delta\ddb\log||S'||'^{2})
-k\chi(\epsilon^{2}+||S||^{2}),
\end{align*}
where $\dr'_{t,\epsilon}=\dr_{t,\epsilon}+\delta\ddb\log||S'||'^{2}.$ By maximal principle again, as
$$\log\frac{\dr'^{n}_{t,\epsilon}(||S||^{2}+\epsilon^{2})^{1-\beta}}{\Omega||S'||'^{2\delta}}-k\chi(\epsilon^{2}+||S||^{2})$$ is uniformly
bounded from below, we have that $\varphi_{\epsilon}\geq-C_{\delta}+\delta\ddb\log||S'||'^{2}.$
Now let's consider the estimate of $\dot{\varphi}_{\epsilon}.$  Denote $\rho=-Ric(\Omega)+(1-\beta)R(||\cdot||)$ again and recall in
\cite{TZ}, we have that
\begin{align*}
  &(\dt-\Delta)(e^{t}\dot{\varphi}_{\epsilon}-\dot{\varphi}_{\epsilon}-\varphi_{\epsilon}-nt)\\=&-tr_{\dr}(\dr_{0}-\rho)+e^{-t}tr_{\dr}
(\dr_{0}-\rho)+\dot{\varphi}_{\epsilon}-\dot{\varphi}_{\epsilon}+n-tr_{\dr}(e^{-t}\dr_{0}+(1-e^{-t})\rho\\&+k\ddb\chi)-n=
-tr_{\dr}(\dr_{0}+k\ddb\chi)=-tr_{\dr}\dr_{0,\epsilon}<0,
\end{align*}
then by maximal principle, for $t>t_{0}>0,$ we have $\dot{\varphi}_{\epsilon}\leq Cte^{-t}.$ On the other hand, consider
$$H=\varphi_{\epsilon}+(1-e^{t-T})\dot{\varphi}_{\epsilon}-\delta\ddb\log||S'||'^{2},$$ for $t<T$ and compute that
$$(\dt-\Delta)H=-n+tr_{\dr}(\rho+e^{-T}(\dr_{0}-\rho)+k\ddb\chi)=-n+tr_{\dr}\dr'_{T,\epsilon}.$$ We know that for $t\infty,$ H tends to
positive infinity near E, so the minimal point of H is attained away from E, where $-n+tr_{\dr}\dr'_{T,\epsilon}\leq 0,$ by maximal
principle. Then we have $\dr\geq\frac{\dr'_{T,\epsilon}}{n}.$ At the minimal point of H, we have that
\begin{align*}
 H&=\varphi_{\epsilon}+(1-e^{t-T})(\log\frac{\dr^{n}(||S||^{2}+\epsilon^{2})^{1-\beta}}{\Omega}-\varphi_{\epsilon}-k\chi)-\delta\ddb
\log||S'||'^{2}\\&\geq e^{t-T}\varphi_{\epsilon}+(1-e^{t-T})(\log\frac{\dr'^{n}_{T,\epsilon}(||S||^{2}+\epsilon^{2})^{1-\beta}}{n^{n}
\Omega}-k\chi)-\delta\ddb\log||S'||'^{2}\\&\geq-C_{\delta},
\end{align*}
then we have that $$(1-e^{t-T})\dot{\varphi}_{\epsilon}\geq-\varphi_{\epsilon}+\delta\ddb\log||S'||'^{2}-C_{\delta}\geq-C_{\delta}.$$
Now let T tend to infinity, we have that $\dot{\varphi}_{\epsilon}\geq-C_{\delta}.$
  When we begin to study Laplacian estimate for $\varphi_{\epsilon},$ we note that the evolution equation for $\log\,tr_{\dr_{0,\epsilon}}
\dr$ is different from before only by a constant. So do the same computation we can derive that $$tr_{\dr_{0,\epsilon}}\dr\leq\frac{C}
{||S'||^{2\alpha}}$$ again. By the argument in \cite{LZ} again we can derive local high order estimate for $\varphi_{\epsilon}$ outside
the base locus E and the divisor D that $$||\varphi_{\epsilon}||_{C^{k}([0,\infty)\times K)}\leq C(k,K)$$ for any $K\subset\subset M
\setminus(D\bigcup E).$\\

        Similar to last two sections, we can choose a sequence of compact sets to approximate $M\setminus(D\bigcup E),$ and a sequence
$\varphi_{\epsilon_{i}},$ by diagonal method, such that $\varphi_{\epsilon_{i}}$ converges to a function $\varphi(t)$ on $[0,\infty)\times
M\setminus(D\bigcup E)$ in $C_{loc}^{\infty}$ sense, which is smooth on $M\setminus(D\bigcup E).$ By the same argument as the last section
we know that this convergence is also in the sense of currents globally. Similarly, as time tends to infinity, $\dr_{\varphi}=\overline
{\dr_{t}}+\ddb\varphi(t)$ is smooth on $M\setminus(D\bigcup E)$ where E is the intersection of all base divisors, i.e. stable base
locus of the twisted canonical bundle, and conic along $D\setminus E$ with cone angle $2\pi\beta.$ As the flow exists forever, we hope to
analyse the limit behavior of $\varphi(t)$ and $\dr_{\varphi}.$ As $\dot{\varphi}_{\epsilon}\leq Cte^{-t}k$ we can see that near the
infinite time, $\varphi_{\epsilon}$ is nonincreasing. Note that we also have $\varphi_{\epsilon}\geq-C_{\delta}+\delta\ddb\log||S'||'^{2},$
which indicates that $\varphi_{\epsilon}$ converges uniformly on arbitrary compact set $K\subset\subset M\setminus(D\bigcup E).$ And we
just see that $\varphi_{\epsilon}$ is nonincreasing near the infinite time, so we can conclude that $\dot{\varphi}_{\epsilon}$ converges
to 0 in $ M\setminus(D\bigcup E).$ And we find that convergence properties here are independent of $\epsilon,$ we can see that $\varphi(t)$
converges uniformly on each compact set $K\subset\subset M\setminus(D\bigcup E),$ and $\dot{\varphi}$ converges to 0 in $ M\setminus
(D\bigcup E).$ Then recall the flow equation \eqref{eq:ma-nef} that the right hand side of it tends to 0. Now set $\overline{\dr}_{\infty}
=\overline{\dr}(\infty)=-Ric(\Omega)+(1-\beta)R(||\cdot||)+k\ddb||S||^{2\beta},$ and $\varphi_{\infty}$ as the limit of $\varphi(t),$ write
 $\dr_{\infty}=\overline{\dr}_{\infty}+\ddb\varphi_{\infty}.$ By \eqref{eq:ma-nef}, we conclude that
\begin{align*}
 Ric(\dr_{\infty})&=Ric(\Omega)-\ddb\varphi_{\infty}-k\ddb||S||^{2\beta}+\log||S||^{2(1-\beta)}\\&=-\dr_{\infty}+R(||\cdot||)+\ddb\log
||S||^{2(1-\beta)}=-\dr_{\infty}+2\pi(1-\beta)[D],
\end{align*}
which means that outside the stable base locus, the current $\dr_{\varphi}$ tends to a conical \ka-Einstein metric, with cone angle $2\pi
\beta$ along $D\setminus E$ as t tends to infinity. Moreover, $C^{2,\alpha}$-estimate still holds.
    Finally, as \cite{TZ}, we can prove that the limit current is independent of the choice of the initial conic metric. First, by the same
argument as in the proof of theorem \ref{mainthm1}, we can see that the solution is independent of the choice of volume form $\Omega.$ Now
fix $\Omega,$ and see whether different initial conic \ka\ metrics matter. Recall that for any $\delta\in(0,a),$ we have the estimates
$$-C_{\delta}+\delta\ddb\log||S'||'^{2}\leq\varphi_{\epsilon}\leq C,\qquad-C_{\delta}+\delta\ddb\log||S'||'^{2}\leq\dot{\varphi}_{\epsilon}
\leq C,$$ which indicates that $\varphi_{\epsilon}+\dot{\varphi}_{\epsilon}\leq C.$ And we also know that $\dot{\varphi}_{\epsilon}$ tends
to 0 as t tends to infinity. Then we obtain that
\begin{align*}
 \int_{M}\frac{e^{\varphi_{\epsilon,\infty}+k\chi}}{(||S||^{2}+\epsilon^{2})^{1-\beta}}\Omega&=\lim_{t\to\infty}\int_{M}
\frac{e^{\varphi_{\epsilon}+\dot{\varphi}_{\epsilon}+k\chi}}{(||S||^{2}+\epsilon^{2})^{1-\beta}}\Omega\\&=\lim_{t\to\infty}\int_{M}(\dr
_{t}+\ddb(\varphi_{\epsilon}+k\chi))^{n}=\lim_{t\to\infty}\int_{M}\dr_{t}^{n}.
\end{align*}
     Now suppose we have two initial conic metrics $\dr^{*}=\dr_{0}+k\ddb||S||^{2\beta}$ and $\dr'^{*}=\dr'_{0}+k'\ddb||S||^{2\beta}.$
Without loss of generality, we can suppose that $\dr_{0}<\dr'_{0}$ and $k<k',$ otherwise we can replace $\dr'^{*}$ by $\dr^{*}+\dr'^{*},$
then we can write the following flow equations:
\begin{align*}
 &\dt(\varphi_{\epsilon}+k\chi)=\log\frac{(\dr_{t}+\ddb(\varphi_{\epsilon}+k\chi))^{n}(||S||^{2}+\epsilon^{2})^{1-\beta}}{\Omega}
-(\varphi_{\epsilon}+k\chi)\\&\dt(\varphi'_{\epsilon}+k'\chi)=\log\frac{(\dr'_{t}+\ddb(\varphi'_{\epsilon}+k'\chi))^{n}(||S||^{2}+
\epsilon^{2})^{1-\beta}}{\Omega}-(\varphi'_{\epsilon}+k'\chi),\\&\varphi_{\epsilon}(\cdot,0)=\varphi'_{\epsilon}(\cdot,0)=0,
\end{align*}
where $\dr'_{t}=e^{-t}\dr'_{0}+(1-e^{-t})(-Ric(\Omega)+(1-\beta)R(||\cdot||))=\dr_{t}+e^{-t}(\dr'_{0}-\dr_{0}).$ Take the difference, and
denote $\varphi_{\epsilon,k}=\varphi_{\epsilon}+k\chi$ and $\varphi'_{\epsilon,k'}=\varphi'_{\epsilon}+k'\chi,$ we have that
\begin{align*}
\dt(\varphi_{\epsilon,k}-\varphi'_{\epsilon,k'})=&\log\frac{(\dr_{t}+\ddb\varphi_{\epsilon,k})^{n}}{(\dr_{t}+\ddb\varphi_{\epsilon,k}
+e^{-t}(\dr'_{0}-\dr_{0})-\ddb(\varphi_{\epsilon,k}-\varphi'_{\epsilon,k'}))^{n}}\\&-(\varphi_{\epsilon,k}-\varphi'_{\epsilon,k'}),
\end{align*}
and $(\varphi_{\epsilon,k}-\varphi'_{\epsilon,k'})(\cdot,0)=(k-k')\chi\leq 0.$ By maximal principle, we have $\varphi_{\epsilon,k}-
\varphi'_{\epsilon,k'}\leq 0$ at any time, so at infinity, we have that $\varphi_{\epsilon,\infty}+k\chi\leq\varphi'_{\epsilon,\infty}
+k'\chi.$ However, as $\int_{M}\frac{e^{\varphi_{\epsilon,\infty}+k\chi}}{(||S||^{2}+\epsilon^{2})^{1-\beta}}\Omega=\lim_{t\to\infty}
\int_{M}\dr_{t}^{n}=\int_{M}\frac{e^{\varphi'_{\epsilon,\infty}+k'\chi}}{(||S||^{2}+\epsilon^{2})^{1-\beta}}\Omega,$ we can obtain that
$\varphi_{\epsilon,\infty}+k\chi=\varphi'_{\epsilon,\infty}+k'\chi.$ Take the limit as $\epsilon$ tends to 0, we complete the proof for
the uniqueness of the limit conical \ka-Einstein current, which also completes the proof for theorem \ref{mainthm3}.

\section{Further remarks}

       Until now we proved three main theorems under the assumption of one irreducible divisor. Now we discuss how to generalize 
our proofs to the case of simple normal crossing divisors briefly. By \cite{GP}, we can add up all the approximation metrics to all 
divisors, and all the properties are preserved. In Laplacian estimate, we note that in \cite{GP}, they proved that the approximation 
metric has a lower bound for bisectional curvature, which allows us to obtain the Laplacian estimate. The proof of $C^{2,\alpha}$-estimate
is essentially the same, see \cite{Sh}.

        Note that in \cite{ST3}, unnormalized \ka-Ricci flow can smooth initial \ka\ metrics with rough and degenerate datas. So naturally, 
we hope unnormalized \ka-Ricci flow can play similar roles in the study of conic \ka\ metrics and we will discuss this phenomenon in the 
future.

\end{document}